\documentclass[12pt]{amsart}

\usepackage[shortlabels]{enumitem}
 
\usepackage{amstext,amssymb,amsmath,stmaryrd,amsthm}

\usepackage{enumitem}

\usepackage{geometry} 
\usepackage{color}
\usepackage{hyperref}
\makeindex

\numberwithin{equation}{section}

\newtheorem{theorem}{Theorem}[section]

\newtheorem{Lemma}{Lemma}[section]

\usepackage{amssymb}
\usepackage{amsmath}
\usepackage{color}
\usepackage{graphicx}
\usepackage{comment}
\usepackage{float}
\usepackage{hyperref}

\newcommand{\R}{\mathbb{R}}

\newcommand{\E}{\mathbb{E}}

\newcommand{\Ge}{\mathcal{L}}

\begin{document}

\title[Asymptotics for the ERW]{Asymptotic Analysis of the Elephant Random Walk}

\author{Cristian F. Coletti and Ioannis Papageorgiou}

\date{}

\address{
\newline
Cristian F. Coletti
\newline
UFABC - Centro de Matem\'atica, Computa\c{c}\~ao e Cogni\c{c}\~ao
\newline
Avenida dos Estados, 5001- Bangu - Santo André - São Paulo, Brasil.
\newline
e-mail:  cristian.coletti@ufabc.edu.br
\newline
\newline
Ioannis Papageorgiou
\newline
UFABC - Centro de Matem\'atica, Computa\c{c}\~ao e Cogni\c{c}\~ao
\newline
Avenida dos Estados, 5001- Bangu - Santo André - São Paulo, Brasil.
\newline
e-mail: papyannis@yahoo.com
}



\begin{abstract} 
In this work we study asymptotic properties of a long range memory random walk known as elephant random walk.  First we prove  recurrence and positive recurrence for the elephant random walk. Then, we establish the transience regime of the model. Finally, under the Poisson Hypothesis, we study the replica mean field limit for this random walk and we obtain an upper bound for the expected distance of the walker from the origin.
\end{abstract}

\date{}
\maketitle 
 
\section{Introduction.}

The asymptotic behavior of random walks with long range memory has been extensively studied over the last years. In particular, the so called elephant random walk (ERW) has raised a considerable interest in the last four years. The ERW was introduced by G. Sch\"utz and S. Trimper \cite{GT} as an example of a non-Markovian process  where it is possible to compute exactly the mean and the variance of the random walk and which exhibits a phase transition from diffusive to superdiffusive behaviour. Independently, \cite{Baur16} and \cite{CGS} proved a strong law of large number and a central limit theorem for the ERW. Then, Bercu \cite{Bercu} obtained some refinements on the asymptotic behavior of the ERW. Indeed, most of the related work on limit theorems of random walks with memory can be subdivided into two categories: the study of limit theorems such as law of large numbers, central limit theorems and invariance principles, see for instance \cite{Baur16}, \cite{Bercu}, \cite{Bertoin}, \cite{CGS}, \cite{CGS2}, \cite{va} and references therein; and hypergeometric identities arising from this kind of processes, see \cite{Bercu1}.

In this paper we focus on the study of recurrence-transience properties for the elephant random walk  as well as on the study of its replica mean field limit under the Poisson Hypothesis. The Poisson hypothesis is an assumption introduced by Kleinrock \cite{K} to justify that some approximations to a given stochastic process become exact in the limit. In our context assuming the Poisson hypothesis amounts to asymptotic independence between replicas.

The paper is organized as follows. In section \ref{defmr} we define the model and state the main results of this work. In section \ref{rectran}  we establish the recurrence (respectively transience) property for the ERW. We also discuss the positive recurrence of our model. Finally, in section \ref{Kanto},  we study the replica mean field limit for the elephant random walk and we obtain an upper bound for the expected distance of the walker from the origin.

\section{Definition of the ERW and main results} \label{defmr}

The ERW is defined as follows. The walk starts at $X_0:=0$ 
at time $n=0$. At each discrete time step the elephant moves one step to the right or to the left respectively, so 
\begin{equation}
X_{n+1} = X_n + \eta_{n+1}
\end{equation}
where $\eta_{n+1} = \pm 1$ is a random variable. The memory consists of the set of random variables $\eta_{n^{\prime}}$
at previous time steps which the elephant remembers as follows:

\noindent $(D_1)$ At time $n+1$ a number $n^{\prime}$ from the set $\{1,2, \ldots , n\}$ is chosen according to a probability mass function $\phi_n$.

\noindent $(D_2)$ $\eta_{n+1}$ is determined stochastically by the rule 
\begin{align*}
\eta_{n+1} = \eta_n^{\prime} \, \text{ with probability } \, p \, \text{ and } \, \eta_{n+1} = -\eta_n^{\prime} \, \text{ with probability } 1-p.
\end{align*}

\noindent $(D_3)$ The elephant starting at $X_0$ moves to the right with probability $r$ and to the left with probability $1-r$, i.e.,
\begin{align*}
\eta_{1} = 1  \, \text{ with probability } \, r \, \text{ and } \, \eta_{1} = -1 \, \text{ with probability } 1-r.
\end{align*}

\noindent It is obvious from the definition that
\begin{equation}
X_n = \sum_{k=1}^n \eta_k.
\end{equation}
Let $\mathbb{P}$ denotes the law of the ERW departing from the origin at time $0$ and let $\mathbb{E}$ denotes the corresponding expectation operator.

\subsection{The uniform case}
In this paper we focus on the case where $\phi_n(i)=\frac{1}{n}$ which we call the uniform case. A simple computation yields 
\begin{equation}\label{conditional}
\mathbb{P}[\eta_{n+1} = \eta|\eta_1, \ldots, \eta_n] = \frac{1}{2n} \sum_{k=1}^n \left[1+\left(2p-1\right)\eta_k \eta\right] \ \mbox{for} \ n \geq 1,
\end{equation}
where $\eta = \pm 1$. For $n=0$ we get in accordance with rule $(D_3)$
\begin{equation} \label{firststep} 
\mathbb{P}[\eta_1 = \eta] = \frac{1}{2}  \left[1+\left(2r-1\right)\eta\right]
\end{equation}
and
\begin{equation}
\mathbb{E}[\eta_1] = 2r- 1.
\end{equation}
The conditional expectation of the increment $\eta_{n+1}$ given its previous history is given by
\begin{equation}\label{conditional2} 
\mathbb{E}[\eta_{n+1} |\eta_1, \ldots, \eta_n] = (2p-1) \frac{X_n}{n} \ \mbox{for} \ n \geq 2.
\end{equation}
A straightforward computation using equations (\ref{conditional}) and (\ref{firststep}) gives

\begin{eqnarray} \label{kernel1}
\mathbb{P}[X_{n+1}=x + y | X_n=x] = \left\{
        \begin{array}{ll}
            \frac{x y (2p-1)+n}{2n} & \quad n \geq 1 \\
            r \delta_1(y) + (1-r) \delta_{-1}(y) & \quad n = 0
        \end{array}
    \right.
\end{eqnarray}

where $y \in \{-1, 1\}$ and $\delta_n (.)$ is the Dirac measure centred on some fixed point $n$. Therefore, $\pi_n(x,x+y) = \mathbb{P}[X_{n+1}=x+y | X_n=x]$ with $y \in \{-1,1\}$ defines a temporal inhomogeneous one-step transition kernel on $\mathbb{Z}$.

\subsection{The generator}
Let $\mathcal{F}_b\left(\mathbb{Z}\right)$ denote the collection of all bounded measurable functions $f: \mathbb{Z} \rightarrow \mathbb{R}$. The generator at time $n$ is the linear operator $\mathcal{L}_n : \mathcal{F}_b\left(\mathbb{Z}\right) \rightarrow \mathcal{F}_b\left(\mathbb{Z}\right)$ defined by

\begin{equation}\label{thegenerator}
 \left(\mathcal{L}_n f\right)(x) := \sum_{y \in \{-1,1\}} \left(f(x+y) - f(x)\right) \pi_n(x,x+y)
\end{equation}
for any $f \in \mathcal{F}_b\left(\mathbb{Z}\right)$.

From now on we refer to any $f \in \mathcal{F}_b\left(\mathbb{Z}\right)$ as a test function. The following lemma is crucial in what follows. Its proof follows from a simple computation and so is omitted.

\begin{Lemma} \label{generator} Let $\left(X_n\right)_{n \geq 0}$ be the ERW with full memory. Then, for any $n \geq 1$ and for any test function $f$ its generator takes the form 
\begin{eqnarray*}
\mathcal{L}_n f (x ) =  \frac{x (2p-1)+n}{2n} \left( f ( x + 1  ) - f(x) \right) + \frac{x (1-2p)+n}{2n} \left( f ( x - 1 ) - f(x) \right).
\end{eqnarray*}
\end{Lemma}

Having determined the exact expression of the generator of the ERW we can now calculate its value for specific test functions. This will be the subject of the following  lemma.

\begin{Lemma} \label{rec_lem} 
Let $\left(X_n\right)_{n \geq 0}$ be the ERW with full memory. Then, for any $n \geq 1$ we have
$$
\mathcal{L}_n \vert x \vert = \left\{
        \begin{array}{ll}
            \frac{2p-1}{n} \vert x \vert & \quad x \neq 0 \\
            1 & \quad x = 0.
        \end{array}
    \right.
$$
\end{Lemma}

\begin{proof}
We only give the proof for $x \neq 0$. Let $h(x)=\vert x \vert $. Then observe that for any $x \neq 0$ and $y \in \{-1,1\}$ we have $\vert x+y \vert-\vert x \vert = y \frac{x}{\vert x \vert}$. Therefore, for any $x \neq 0$ we have
\begin{align*} \nonumber  
\mathcal{L}_n h(x)= &\frac{1}{2n}  \left((x(2p-1)+n) \frac{x}{\vert x \vert}+ (x(1-2p)+n) \frac{-x}{\vert x \vert}   \right) \\ =&   \frac{(2p-1)}{n} h(x) . 
\end{align*} 
\end{proof}

\section{Recurrence and transience} \label{rectran}

Now we present the main result of this section which establishes the recurrence and transience regimes of the ERW. The main obstacle in the study of the properties of recurrence and transience of the elephant random walk comes from the fact that this random walk is temporal inhomogeneous.  We remark that the property of being recurrent or transient for this specific random walk is determined exclusively by the value of $p$.

\begin{theorem} \label{newrec}
Let $\left(X_n\right)_{n \geq 0}$ be the ERW with full memory. Then, if $p \leq 3/4$ the ERW is recurrent.
\end{theorem}

\begin{proof}
We begin by assuming that $p < 3/4$. The critical case $p=3/4$ will be discussed afterwards. We will prove that the event
\begin{equation} \label{rec}
\left[\liminf_{n \rightarrow +\infty} X_n = - \infty, \limsup_{n \rightarrow +\infty} X_n = + \infty \right] 
\end{equation}
occurs a.s.

In \cite{Bercu} (see also \cite{CGS2}) the author proved that 
\begin{equation}
\limsup_{n \rightarrow +\infty} \frac{X_n}{\sqrt{2n \ln \ln n}} = \frac{1}{\sqrt{3-4p}}.    
\end{equation}
Therefore, 
\begin{equation}
 \frac{X_n}{\sqrt{2n \ln \ln n}} \geq \frac{1}{2\sqrt{3-4p}} 
\end{equation}
for infinitely many $n$. 
Since $\sqrt{2n \ln \ln n}$ diverges, it follows that for any $M > 0$ there exist infinitely many $n$ such that
\begin{equation}
X_n \geq M.
\end{equation}
In this way we may conclude that 
\begin{equation}
\limsup_{n \rightarrow +\infty} X_n = + \infty \ \mbox{a.s.}
\end{equation}
In the same work it is proved that
\begin{equation}
\liminf_{n \rightarrow +\infty} \frac{X_n}{\sqrt{2n \ln \ln n}} = - \frac{1}{\sqrt{3-4p}}.    
\end{equation}
It follows immediately that
$\liminf_{n \rightarrow +\infty} X_n = - \infty$ a.s. Then, the event defined in (\ref{rec}) has probability one as claimed. In other words the random walk is recurrent for $p < 3/4$. The proof that the random walk is recurrent for the critical case $p=3/4$ is entirely analogous since for $p=3/4$ we have that
\begin{equation} \label{c1}
\limsup_{n \rightarrow +\infty} \frac{X_n}{\sqrt{2n \ln n \ln \ln \ln n}} = 1
\end{equation}
and
\begin{equation} \label{c2}
\liminf_{n \rightarrow +\infty} \frac{X_n}{\sqrt{2n \ln n \ln \ln \ln n}} = - 1 .
\end{equation}
For a proof of (\ref{c1}) and (\ref{c2}) see  \cite{Bercu} and \cite{CGS2}.
\end{proof}

Now we proceed to show that for $p<1/6$ the ERW is actually positive recurrent. In order to do so we need to introduce some notation. For $m \geq 1$, let $\mathbb{E}_x^m\left(  . \right) := \mathbb{E}_x^m\left[  .  | X_m = x \right]$.

\begin{theorem} \label{positiverecurrent} 
Let $\left(X_n\right)_{n \geq 0}$ be the ERW with full memory and let $\tau_0:=\inf\{n:X_n=0\}$.  If $p < 1/6, m \geq 1$ and $x \neq 0$, then
\begin{eqnarray*}
\E_x^m[\tau_0] \leq \frac{2}{1-6p}\vert x \vert + 1
\end{eqnarray*}
and the ERW with full memory is positive recurrent.
\end{theorem}
 
\begin{proof} 
Let $f(x) = \vert x \vert $ . In lemma \ref{rec_lem} we proved that
\begin{equation} \label{pos_rec}
\mathcal{L}_n f(x) = \frac{(2p-1)}{n} \vert x \vert .
\end{equation}
for any $x \neq 0$. 

Now, let $Z_i = (i+m) |X_{i+m}|$ and $\mathcal{G}_i = \sigma \left(\eta_k , k=1, \ldots , i+m \right)$. Then, it follows from (\ref{pos_rec}) that 
\begin{eqnarray}
\mathbb{E}^m_x \left[ Z_{i+1} - Z_i | \mathcal{G}_i \right] &=& \mathbb{E}^m_x \left[ (i+m+1) |X_{i+m+1}| - (i+m) |X_{i+m}| | \mathcal{G}_i \right] \nonumber \\
&\leq& \frac{(i+m+1)}{i+m} \left(2p-1\right) |X_{i+m}|  + \mathbb{E}^m_x \left[ |X_{i+m}| | \mathcal{G}_i \right] \nonumber \\
&=& \frac{(i+m+1)}{i+m} \left(2p-1\right) |X_{i+m}|  + |X_{i+m}| \nonumber \\
&=& \left(\left(1 + \frac{1}{i+m}\right) \left(2p-1\right) + 1 \right)|X_{i+m}| \nonumber \\
&\leq& \left(\frac{3 \left(2p-1\right)}{2} +1\right) |X_{i+m}| .  \nonumber \\
&=& \frac{6p-1}{2} |X_{i+m}| .  \nonumber 
\end{eqnarray}

Now, for any stopping time $\tau$ and any discrete-time stochastic process $\left(Z_k \right)_{k \geq 0}$, define  
\begin{equation} \label{eqnv}
\tau^{n}:=\min \left\{n,\tau,\inf\{k\geq 0:Z_k\geq n\}\right\}. 
\end{equation}

It follows from the discrete Dynkin's formula and equation \ref{pos_rec} that, for any $x\neq 0$,
\begin{align*}
\E_x^m[Z_{\tau_0^n}]=&\E_x^m[Z_0]+\E_x^m\left[  \sum_{i=1}^{\tau_0^{n}}\left( \E_x^m[Z_i / \mathcal{G}_{i-1} ]-Z_{i-1} \right) \right]
\\  \leq &
\vert x \vert+ \frac{6p-1}{2} \E_x^m\left[  \sum_{i=1}^{\tau_0^{n}}\vert X_{i+m} \vert \right].
\end{align*}
Observe that $Z_{\tau_0^n}\geq 0$ and that $\vert X_i \vert \geq 1$ for $i < \tau_0^n$ in the sum above. Therefore,
\begin{equation*}
\E_x^m[\tau_0^n]\leq \frac{2}{(1-6p)}\vert x \vert + 1.
\end{equation*}
It follows from Fatou's lemma applied to the inequality above that
\begin{eqnarray*}
\E_x^m(\tau_0) &=& \E_x^m(\lim_{n \rightarrow +\infty}\tau^n_0) \\ 
&\leq& \liminf_{n \rightarrow +\infty} \E_x^m(\tau^n_0) \\
&\leq& \frac{2}{(1-6p)}\vert x \vert + 1.
\end{eqnarray*}
This finishes the proof.
\end{proof}

We finish this section by establishing the transient regime for the ERW.

\begin{theorem} 
Let $\left(X_n\right)_{n \geq 0}$ be the ERW with full memory. Then, if $p > 3/4$ the ERW is transient.
\end{theorem}

\begin{proof}
In \cite{Bercu} it was proved that if $p > 3/4$, then the sequence $(M_n)$ defined by $M_n = a_n X_n$ is a martingale with respect to the natural filtration. Here $(a_n)$ is a sequence such that $\lim_{n \rightarrow +\infty} n^{2p-1} a_n= \Gamma(2p)$. Then, there exists $N$ such that for any $n \geq N$, we have
\begin{equation}
\frac{X^2_n}{n} \geq C n^{4p-3} M^2_n
\end{equation}
for some positive constant $C$. Since $p > 3/4$, $n^{4p-3}$ diverges. In order to prove transience it only remains to show that $M^2_n$ is bounded away from zero for values of $n$ large enough. Since $M_n$ converges a.s. to a finite random variable $M$ (\cite{Bercu}, \cite{CGS}), it suffices to have $M \neq 0$. To conclude the proof we appeal to the well known relation between the ERW and P\'olya-types urns, see \cite{Baur16}. The authors observe that for $n \geq 1, X_n$ is equal in distribution to $U_n^1 - U_n^2$ where $U_n = \left(U_n^1 , U_n^2 \right)$ is the composition of a discrete-time urn with two colors. When $p > 3/4$ such urn is called large and its limit distribution is studied in \cite{Chauvin}. Theorems $2.1, 7.2$ and $7.4$ from that paper guarantee that $U_n$, after centering and re-scaling, converges to a random vector admitting a density which in turns guarantees that $M$ admits a density.
\end{proof}

\section{Replica Mean Field Limit for the ERW.}\label{Kanto}
 
We now focus in studying the Replica Mean Field (RMF) limit of the ERW, which allows us to  give an estimate of the expected distance of the walker from the origin 

Recently Baccelli and Taillefumer (see \cite{B-T}) used RMF limits to describe the stationary state of a system with a finite number of neurons. We refer the interested reader to the works \cite{r8},  \cite{r10},  \cite{r49} and  \cite{r60} for related and relevant works.
 
The main idea hidden behind the RMF approach  is to describe the moments of the invariant measure in terms of basic structural elements  of the process. In the case of the ERW the structural element is the probability $p$ of copying the past.
 
The results presented in the current section as well as the techniques used to obtain them follow closely the work in \cite{B-T}. Indeed, we refer the interested reader to \cite{B-T} and references therein for a detailed and comprehensive explanation  of the RMF approach. 

Now we proceed to study the distance of the walker from the origin in relation to time, i.e. $|X_n / n|$. Therefore, and without loss of generality we may assume that all the elephants only move on the non-negative integer lattice. Finally we define the finite  Replica Mean Field system consisting of  $M$ identical copies of an elephant random walk $X_n:=(X_{n_i}^{i})_{i=1}^{M}$ after $n=\sum_{i=1}^{M}n_i$ total jumps, where $n_i$ is the number of jumps given by the $i$-th elephant, $M > 1$ and $i\in \{1, \ldots , M\} $.  If $\eta_k^{i}$ denotes the $k$-th jump given by the $i$-th elephant where $k\in\{1,...,n_i\}$ and   $i\in \{1, ..., M\} $, then we have
\[
X^{i}_{n_i}=\sum_{k=1}^{n_i}\eta^{i}_k   \text{ \   with  \  }  X_0^{i}=0.
\]
Now we define the generator of the $M$-finite Replica Mean Field model. In analogy to the elephant random walk we consider an ERW with a non-homogeneous one-step transition kernel. The generator of the finite $M$-RMF process  $X_{n+1}$    in any test function $ f : \R^M  \to \R $  and $x \in \R_+^M$  takes the form    
\begin{align}
\Ge_n f (x ) = \sum_{i=1}^{M} \xi_i \sum_{j=1}^{M}\psi^{i}(j)\sum_{k=1 }^{n_j}   \phi^j_{n_j}(k) \sum_{y \in \{-1,1\}} p^{i,j}(y) \left[ f ( x + y e_i ) - f(x) \right] ,
\end{align}
where  $n=\sum_{i=1}^{M}n_i, x = (x_1,...,x_M)$ and $\{e_1, \ldots , e_M\}$ is the canonical basis of $\R^M$.

Here $\xi_i$ is the probability that the $i$-th elephant will be chosen to be the next one to make a step, $\phi^{i}(j)$ is the probability to choose the $j$-th elephant to determine how the $i$-th will move  and $\phi^j_{n_j}(k)$ denotes the probability to choose the $k$-th step, $k\leq n_j$, of the $j$-th elephant. Finally, $p^{i,j}(y)$ is the probability that the $i$-th elephant decides to give one step to the right or to the left (depending on wether $y=1$ or $y=-1$) according to rule ($D_2$) applied to the $j$-th elephant.

In this section we consider the following setting.
\begin{enumerate} [label=\roman*)]
\item The elephant that will move next  is chosen uniformly, i.e. $\xi(i)=\frac{1}{M}$.
\item If $j$ is the next elephant that will move, then the elephant  whose path will determine the next move of $j$ is chosen uniformly among the remaining elephants, i.e. $\psi^{i}(j) =\frac{1}{M-1}$ for any $j \neq i$. 
\item The process does not loose any of its memory. The step of the elephant selected in item $(ii)$ is chosen uniformly among all the previous steps   $k \leq n_j$, i.e. $ \phi^j_{n_j}(k)=\frac{1}{n_j}$.
\item The $i$-th elephant copies the step given in the past by the $j$-th elephant with probability $p$.
\end{enumerate}

In order to avoid trivial degeneracies at the beginning of the process and in the counting of the steps, we make the following assumption.

\begin{enumerate} [a)]

\item In order to avoid choosing an elephant that has not moved yet to be the one that will determine the next step of any other elephant we may assume,  without lose of generality, that at the beginning of the process all elephants make an initial jump from their starting position $0$, to $-1$ or $1$, with equal probability $\frac{1}{2}$.
\end{enumerate}
Having established the interaction dynamic between the replicas, we can now determine the generator of the finite  $M$-Replica Mean Field model. This is the subject of the following lemma.

\begin{Lemma} Let $\left(X_n\right)_{n \geq 0}$ be the $M$-Replica Mean Field Model. Then for any $1 \leq i \neq j \leq M$, the generator takes the form
\[
\Ge_n f (x) =  \sum_{i=1}^{M}   \sum_{j=1,j\neq i}^{M}  \sum_{y \in \{-1,1\}} \frac{\pi^j_{n_j}(x_j,x_j+y)}{M(M-1)} \left[ f ( x + y e_i) - f(x) 
\right] ,
\]
where
\begin{align*}
\pi^j_{n_j}(x_j,x_j+y)&= \frac{  x_j y (2p-1 )+ n_j}{2n_j}
\end{align*}
and $y \in \{-1,1\}$.
\end{Lemma}
\begin{proof}
Since we are dealing with the uniform case, it follows from Lemma \ref{generator} that
\begin{align}
\Ge_n f (x ) = \sum_{i=1}^{M} \xi_i \sum_{j=1}^{M}\psi^{i}(j) \sum_{y \in \{-1,1\}} \frac{  x_j y (2p-1 )+ n_j}{2n_j} \left[ f ( x + y e_i ) - f(x) \right] .
\end{align}
The proof follows from the fact that $\xi_i= \frac{1}{M}$ and that $\psi^{i}(j)=\frac{1}{M-1}$ for any $j \neq i$.
\end{proof}

Let $\E^{M}$ denotes the expectation operator induced by the finite $M$-Replica Mean Field model. 

Below we present the main result about the Replica Mean Field limit.

\begin{theorem} 
If $p<1/4$, then
\begin{equation*}
\limsup_{M \rightarrow  \infty}\E^{M}[\frac{ X^i_{n_i}}{n_i} ] \leq \frac{1}{2 (1-2p)}
\end{equation*}
for any $i \geq 1$. Also, if $p \geq 3/4$ and $i \geq 1$, then
\begin{equation*}
\limsup_{M \rightarrow  \infty}\E^{M}[\frac{ X^i_{n_i}}{n_i} ] \leq \frac{1}{2 (2p-1)}
\end{equation*}
\end{theorem}
\begin{proof}

We prove the case $p > 3/4$. The proof of the case $p < 1/4$ follows in the same lines of those of the case $p > 3/4$.

For any $u \geq  0$ and $x=(x_1, \ldots , x_M)$, let $V_u(x)=\prod_{i=1}^M e^{u  x_i}$. We    have 
\begin{align*} 
\Ge_n  [ V_{u }](x)&=\sum_{i=1}^{M}   \sum_{j=1,j\neq i}^{M}  \sum_{y \in \{-1,1\}}  \frac{\pi^j_{n_j}(x_j,x_j+y)}{M(M-1)} \left[  V_{u } ( x + y e_i ) -  V_{u }(x) \right] \\ 
  &= \sum_{j=2}^{M}  \sum_{y \in \{-1,1\}}  \frac{ y x_{n_j}(2p-1 )+n_j}{2n_jM(M-1)} \left[ e^{u y} -  1 \right] \prod_{i=1}^M e^{u  x_i} . 
\end{align*}

Using the equation above we compute the expectation of $V_u\left(X_n\right)$ with respect to $\E^{M}$ where $X_n$ is the RMF model $(X_{n_i}^{i})_{i=1}^{M}$. Then

\begin{align}\label{ode}
\E^{M} \left[ \Ge_n \left[ V_u \right] \left(X_n\right) \right] &=  (2p-1) \left[ e^{ u} -  e^{ -u}  \right] \sum_{j=1}^{M} \E^{M}[Y_j \prod_{i=1}^M e^{u X^i_{n_i}}] \\
&+ (e^{u}+e^{-u} -  2 )(M-1)\E^{M}[ \prod_{i=1}^M e^{u X^i_{n_i}}] \nonumber
\end{align}
where $Y_i := \frac{X^i_{n_i}}{n_i}$. 

Under the Poisson hypothesis we have asymptotic independence between replicas. Then, for any $j \geq 1$ we have

\begin{align*}
\lim sup_{M\rightarrow \infty}\E^{M}[Y_j \prod_{i=1}^M e^{u X^i_{n_i}}] &= \E^{\infty}[Y_je^{u X^j_{n_j}}]\E^{\infty}[\prod_{i=1, i\neq j}^M e^{u X^i_{n_i}}] \\
&= \E^{\infty}[Y_1e^{u X^1_{n_1}}]\E^{\infty}[\prod_{i = 2}^M e^{u X^i_{n_i}}]
\end{align*}
where $\E^{\infty}[f]:=\lim sup_{M\rightarrow \infty}\E^{M}[f]$.

Passing to the limit when $M\rightarrow \infty$ in  (\ref{ode}) leads to
\begin{align*}
\limsup_{M\rightarrow \infty} \E^{M} \left[ \Ge_n \left[ V_u \right] \left(X_n\right) \right] &=       (2p-1) \left[ e^{ u} -  e^{ -u}  \right]  \E^{\infty}[Y_1e^{u X^1_{n_1}}]\E^{\infty}[\prod_{i =2}^M e^{u X^i_{n_i}}] \\
&+    ( e^{u}+e^{-u} -  2 )\E^{\infty}[\prod_{i=1}^M e^{u X^i_{n_i}}]
\end{align*}
where $u\geq 0$. Observe that for any $n, M \geq 1$ we have
\begin{equation*}
\prod_{i=1}^M e^{u X^i_{n_i}} \leq e^{u n} \ \mbox{a.s.} \ 
\end{equation*}
where $n = n_1 + \ldots + n_M$. We can write 
  \begin{align*}
\E^{\infty}[Y_1e^{u X^1_{n_1}}]\E^{\infty}[\prod_{i=2}^M e^{u X^i_{n_i}}] &= \frac{ 2-e^{u}-e^{-u}  }{(2p-1)(e^{ u} -  e^{ -u} )}\E^{\infty}[\prod_{i=1}^M e^{u X^i_{n_i}}] \\
&+ \limsup_{M\rightarrow \infty} \frac{\E^{M}[\Ge_n [V_u]\left(X_n\right)]}{(2p-1)(e^{ u} -  e^{ -u})}.
\end{align*}
By passing to the limit when $u$ goes to $0$ from the right we obtain 
\begin{align*}
\limsup_{M \rightarrow  \infty}\E^{M}[\frac{ X_{n_i}}{n_i} ] &= \lim_{u\rightarrow 0}  \frac{ 2+e^{-u}-e^{u} }{e^{ u} -  e^{ -u} } \\ 
&+ \lim_{u\rightarrow 0}\frac{\limsup_{M\rightarrow \infty}\E^{M}[\Ge_n [V_u]\left(X_n\right)]}{(2p-1)(e^{ u} -  e^{ -u})} .
\end{align*}
The first term on the rhs of the inequality above equals zero. The second term equals
\begin{align*}
\lim_{u\rightarrow 0}\frac{\limsup_{M\rightarrow \infty}\E^{M}[\Ge_n [V_u]\left(X_n\right)]}{(2p-1)(e^{ u} -  e^{ -u})} &= \limsup_{M \rightarrow +\infty} \frac{\E^{M} \left[ \Ge_n \left(\sum_{i=1}^M X^i_{n_i} \right)\right]}{2(2p-1)} \\
&\leq \frac{1}{2(2p-1)}.
\end{align*}

The last inequality above follows from the fact that the generator inside the expectation is between $-1$ and $1$. Therefore
\begin{equation*}
\limsup_{M\rightarrow \infty}\E^{M}[\frac{X^i_{n_i}}{n_i} ]\leq\frac{1}{2(2p-1) }.
\end{equation*}
\end{proof}

\section*{Acknowledgements}
The authors thank an anonymous referee for a careful reading of a previous version of this work. Her/his comments and observations definitively help to improve the quality of the presentation. This work was partially supported by FAPESP (2017/10555-0). The authors are grateful to the organizers of the 2019 Brazilian School of Probability where this work began.

\end{document}